\documentclass[12pt]{amsart}
\usepackage{geometry}
\usepackage{amsxtra,amssymb,amsthm,amsmath,amscd,mathrsfs, epsfig}
\usepackage{amscd, amsmath, mathrsfs, amssymb, amsthm, amsxtra, bbding, epsfig, eucal, graphicx, latexsym, mathrsfs, mathbbol, bbold}
\usepackage[all]{xy}

\usepackage{tikz-cd}

\usepackage{tikz}
\usepackage[normalem]{ulem}
\usepackage{soul}
\usepackage{color}

\setstcolor{red}

\setstcolor{blue}
\makeatletter
\@namedef{subjclassname@2010}{%
 \textup{2010} Mathematics Subject Classification}
\makeatother

\def \N {{\mathbb N}}
\def \P {{\mathbb P}}

\def \Z {{\mathbb Z}}

\def\e{{\rm e}}
\def\i{{\rm i}}

\def\stacksum#1#2{{{\scriptstyle #1}}\atop {{\scriptstyle #2}}}
\def\le{\leqslant}
\def\leq{\leqslant}
\def\ge{\geqslant}
\def\geq{\geqslant}

\theoremstyle{plain}
\newtheorem{theorem}{Theorem}

\newtheorem{lemma}{Lemma}[section]
\newtheorem{corollary}{Corollary}

\theoremstyle{remark}

\theoremstyle{definition}

\numberwithin{equation}{section}

\begin{document}

\vskip 5mm

\title[On modular signs of half integral weights]
{On modular signs of half integral weights}
\author{Bin Chen, Jie Wu \& Yichao Zhang}

\address{
Bin Chen
\\
School of Mathematics
\\
Shandong University
\\
Jinan, Shandong 250100
\\
China.
School of Mathematics and Statistics
\\
Weinan Normal University
\\
Weinan, Shaanxi 714099
\\
China.
}
\email{13tjccbb@tongji.edu.cn}

\address{%
Jie Wu\\
CNRS LAMA 8050\\
Laboratoire d'analyse et de math\'ematiques appliqu\'ees\\
Universit\'e Paris-Est Cr\'eteil\\
61 avenue du G\'en\'eral de Gaulle\\
94010 Cr\'eteil Cedex\\
France
}
\email{jie.wu@math.cnrs.fr}

\address{%
Yichao Zhang
\\
School of Mathematics and Institute for Advanced Study in Mathematics of HIT
\\
Harbin Institute of Technology
\\
Harbin 150001
\\
China
}
\email{yichao.zhang@hit.edu.cn}

\date{\today}


\begin{abstract}
In this paper, we consider the first negative eigenvalue of eigenforms of half-integral weight $k+1/2$
and obtain an almost type bound.
\end{abstract}

\maketitle

\addtocounter{footnote}{1}

\vskip 10mm

\section{Introduction}

Let $k\geq 3$ be an integer and denote by $\mathfrak{S}_{k+1/2} = \mathfrak{S}_{k+1/2}(4)$ the space of cusp forms of half-integral weight $k+\frac{1}{2}$ on the congruence subgroup $\Gamma_0(4)$. Let $\mathfrak{S}^+_{k+1/2}$ be \emph{Kohnen's plus space} in $\mathfrak{S}_{k+1/2}$ and $\mathfrak{S}^{+,*}_{k+1/2}$ be a basis of Hecke eigenforms of $\mathfrak{S}^+_{k+1/2}$.
For  $\mathfrak{f}\in\mathfrak{S}^{+,*}_{k+1/2}$, let $\mathfrak{a}_\mathfrak{f}(n)$ be its $n$-th Fourier coefficient. For a positive square-free integer $t$ with $\mathfrak{a}_\mathfrak{f}(t)\neq 0$, set $\mathfrak{a}^*_\mathfrak{f}(n^2)=\mathfrak{a}_\mathfrak{f}(t)^{-1}\mathfrak{a}_\mathfrak{f}(tn^2)n^{-k+\frac{1}{2}}$, which is independent of $t$ by Shimura's theory \cite{Shimura1973}. See Section 2 for some basics on half-integral weight modular forms.

In this paper, we will investigate sign changes of the sequence $\{{\mathfrak a}_{\mathfrak{f}}^*(n^2)\}_{n\ge 1}$.
This problem has received much attention 
\cite{LauWu2009, KohnenLauWu2013, LauLiuWu2013, LauRoyerWu2016, ChenWu2016, JLLRW2018}.
In particular, denoting by $n_{\mathfrak{f}}$ the smallest integer $n$ such that
\begin{equation}\label{def:nf}
{\mathfrak a}_{\mathfrak{f}}^*(n^2)<0.
\end{equation}
Recently Chen and Wu \cite{ChenWu2016} proved, by developing the method of \cite{KLSW2010}, that
for each $\mathfrak{f}\in \mathfrak{S}_{k+1/2}^{+, *}$, we have
$$
n_{\mathfrak{f}}\ll k^{9/10}
$$
uniformly for all $k\ge 3$, where the implied constant is absolute.
The aim of this paper is to improve this bound on average.
Our result is as follows.

\begin{theorem}\label{thm1}
Let $\nu\ge 1$ be an integer and
let $\mathcal{P}$ be a set of prime numbers of positive density in the following sense:
\begin{equation}\label{Hyp:mathcalP}
\sum_{\substack{z<p\le 2z\\ p\in \mathcal{P}}} \frac{1}{p}
\ge \frac{\delta}{\log z} 
\qquad(z\ge z_0)
\end{equation} 
for some constants $\delta>0$ and $z_0>0$.  
Then there are two positive constants $C$ and $c$
such that for any $\{\varepsilon_p\}_{p\in \mathcal{P}}\subset \{-1, 1\}^{\mathcal{P}}$, 
the number of the Hecke eigenforms $\mathfrak{f}\in \mathfrak{S}_{k+1/2}^{+, *}$ satisfying the condition
\begin{equation}\label{thm1:eq1}
\varepsilon_{p}{\mathfrak a}_{\mathfrak{f}}^*(p^{2\nu})>0
\quad\text{for}\quad
C\log k< p \leq 2C\log k
\end{equation}
is bound by
\begin{equation}\label{thm1:eq2}
\ll k \exp(-c(\log k)/\log_2k),
\end{equation}
where the implied constant are absolute and $\log_2 := \log\log$.
\end{theorem}

For $\mathfrak{f}\in \mathfrak{S}_{k+1/2}^{+, *}$, denote by $n_{\mathfrak{f}}^*$ the smallest prime number $p$ such that
\begin{equation}\label{def:nf*}
{\mathfrak a}_{\mathfrak{f}}^*(p^2)<0.
\end{equation} We have trivially 
$$
n_{\mathfrak{f}}\le n_{\mathfrak{f}}^*
$$
for all $\mathfrak{f}\in \mathfrak{S}_{k+1/2}^{+, *}$.
Setting $\mathcal{P}=\P$ (set of all prime numbers), $\varepsilon_{p}=1$ for all $p\in \P$ and $\nu=1$ in Theorem \ref{thm1}, 
we immediately obtain the following result.

\begin{corollary}\label{cor1}
There is an absolute positive constant $c$ such that
$$
n_{\mathfrak{f}}^*\ll \log k
$$
for all $\mathfrak{f}\in \mathfrak{S}_{k+1/2}^{+, *}$, except for $\mathfrak{f}$ in an exceptional set with
$$
\ll  k \exp(-c(\log k)/\log_2k)
$$
elements, where the implied constants are absolute.
\end{corollary}

In the opposite direction, we have the following result.

\begin{theorem}\label{thm2}
There are two absolute positive constants $c_1$ and $c_2$ such that
$$
\big|\big\{\mathfrak{f}\,\in \mathfrak{S}_{k+1/2}^{+, *} \,:\, n_{\mathfrak{f}}^*\ge c_1\sqrt{(\log k)\log_2k}\big\}\big|
\gg k\exp\big(-c_2\sqrt{(\log k)/\log_2k}\big),
$$
provided that $k$ is large enough.
Here the implied constant is absolute.
\end{theorem}

Our approach is rather flexible.
In view of the half-integral weight newform theory \cite{Kohnen1982, MRV1990},
our results could be further generalized to the case of $\mathfrak{S}_{k+1/2}^{+, *}(4N, \chi)$,
where $k\ge 3$ is an integer, $N\ge 1$ is square free, $\chi$ is a quadratic character of Dirichlet and
$\mathfrak{S}_{k+1/2}^{+, *}(4N, \chi)$ is the set of all eigenforms in $\mathfrak{S}_{k+1/2}^{+}(4N, \chi)$
--- Kohnen's plus subspace of cusp forms of half-integral weight $k+1/2$ for $\Gamma_0(4N)$ with character $\chi$.

\vskip 8mm

\section{Shimura correspondence}

In this section, we cover briefly Shimura's theory on half-integral weight modular forms and the Shimura correspondence. Throughout let $k\ge 3$ be an integer and denote by $\mathbb{P}$ the set of prime numbers. 

Denote by $\mathcal{H}_{2k} = \mathcal{H}_{2k}(1)$ and $\mathfrak{S}_{k+1/2} = \mathfrak{S}_{k+1/2}(4)$
the space of cusp forms of weight $2k$ on the modular group $\text{SL}_2(\Z)$ and that of cusp forms of weight $k+\frac{1}{2}$ on the congruence subgroup $\Gamma_0(4)$, respectively. 
For $f\in \mathcal{H}_{2k}$ and ${\mathfrak{f}}\in \mathfrak{S}_{k+1/2}$, denote their Fourier expansions at infinity by
\begin{equation}\label{deffrakfz}
{f}(z)=\sum_{n\geq 1}a_f(n) \e^{2\pi\i nz},
\qquad 
{\mathfrak{f}}(z)
= \sum_{n\ge 1} {\mathfrak a}_{\mathfrak{f}}(n) \e^{2\pi\i nz}.
\end{equation}
Denote by $\mathfrak{S}_{k+1/2}^+$ the subspace in $\mathfrak{S}_{k+1/2}$ of all forms $\mathfrak{f}$ with
${\mathfrak a}_{\mathfrak{f}}(n)=0$ for all $n$ verifying $(-1)^k n \equiv 2, 3\,({\rm mod}\,4)$.
This subspace is called Kohnen's plus space (cf. \cite{Kohnen1982}).

For each positive integer $n$, there is an Hermitian operator $T_{2k}(n)$, the $n$-th \emph{Hecke operator}, on $\mathcal{H}_{2k}$, and $\{T_{2k}(n)\colon n\geq 1\}$ has the structure of a commutative algebra, the \emph{Hecke algebra} on $\mathcal{H}_{2k}$. Consequently, there is a basis $\mathcal{H}^*_{2k}$ of common eigenfunctions to all of $T_{2k}(n)$ such that $T_{2k}(n)f=a_f(n)f$ 
for each $f\in\mathcal{H}_{2k}^*$. Elements of $\mathcal{H}_{2k}^*$ are called \emph{normalized Hecke eigenforms} 
in $\mathcal{H}_{2k}$.

On the other hand, for each positive integer $n$, Shimura \cite{Shimura1973} introduced the $n^2$-th Hecke operator $T_{k+1/2}(n^2)$ on $\mathfrak{S}_{k+1/2}$, and the Hecke algebra of all Hecke operators is again commutative. Kohnen considered the restriction $T_{k+1/2}^+(n^2)$ of  $T_{k+1/2}(n^2)$ to his plus space $\mathfrak{S}_{k+1/2}^{+}$, and proved that $T_{k+1/2}^+(n^2)$ becomes an Hermitian operator. Therefore, there exists a basis of common eigenfunctions to all operators $T_{k+1/2}^+(n^2)$ in $\mathfrak{S}_{k+1/2}^{+}$. We fix such a basis and denote it by $\mathfrak{S}_{k+1/2}^{+,*}$. Note that the leading coefficient $\mathfrak{f}$ is not $\mathfrak{a}_\mathfrak{f}(1)$ in general, and normalizing the leading coefficient to be $1$ may lose the algebraicity of the Fourier coefficients. As a consequence, unlike the case of integral weight, there is no canonical choice for $\mathfrak{S}_{k+1/2}^{+,*}$, but it causes no problems for our purpose. 

Discovered by Shimura \cite{Shimura1973}, there exist liftings from Hecke eigenforms of half-integral weight to Hecke eigenforms of integral weight, the Shimura correspondence. Then Shintani \cite{Shintani1975} introduced the method of theta lifting and established the correspondence in the opposite direction, that is, from forms of integral weight to forms of half-integral weight. 
On $\mathfrak{S}_{k+1/2}^{+}$, Kohnen \cite[Theorem 1]{Kohnen1980} built an isomorphism between $\mathfrak{S}_{k+1/2}^{+}$ and $\mathcal{H}_{2k}$ as Hecke modules. 
So in particular, as $k\to\infty$,
\begin{equation}\label{dimension}
|\mathfrak{S}_{k+1/2}^{+, *}|
= |\mathcal{H}_{2k}^*| 
= \tfrac{1}{6}k + O(k^{1/2}).
\end{equation}

Now let us explain the Shimura correspondence explicitly. Fix a positive square-free integer $t$ 
and the Shimura correspondence $\mathcal{S}_t$ is defined as follows: 
For each ${\mathfrak{f}}\in \mathfrak{S}_{k+1/2}^+$, $f_t:=\mathcal{S}_t(\mathfrak{f})$ has Fourier expansion at $\infty$
\begin{equation}\label{def:ftz}
f_t(z)
= \sum_{n\ge 1} a_{f_t}(n) \e^{2\pi\i nz},
\end{equation}
where 
\begin{equation}\label{relation:aft-frakaft}
a_{f_t}(n)
:= \sum_{d\mid n} \chi_t(d) d^{k-1}
{\mathfrak a}_{\mathfrak{f}}\bigg(t\frac{n^2}{d^2}\bigg),
\qquad 
\chi_t(d) := \bigg(\frac{(-1)^kt}{d}\bigg).
\end{equation}
Here $\left(\frac{\cdot}{\cdot}\right)$ denotes the Kronecker symbol, an extension of Jacobi's symbol to all integers
(see \cite{Shimura1973}, page 442). Then $f_t\in \mathcal{H}_{2k}$.
Furthermore, if ${\mathfrak{f}}$ is a Hecke eigenform with eigenvalue $\omega_p$ for $T_{k+1/2}^+(p^2)$, then we may choose $t$ with $\mathfrak{a}_\mathfrak{f}(t)\neq 0$, and $\mathcal{S}_t(\mathfrak{f})$ becomes a Hecke eigenform in $\mathcal{H}_{2k}$ with leading coefficient $\mathfrak{a}_\mathfrak{f}(t)$. Actually, 
\begin{eqnarray}\label{sl}
f(z) := \mathfrak{a}_{\mathfrak{f}}(t)^{-1} f_t(z)\in\mathcal{H}^*_{2k}
\end{eqnarray}
and the $L$-function $L(s,f) = \prod_{p\in\mathbb{P}} (1-\omega_pp^{-s}+p^{2k-1-2s})^{-1}$. 
It follows that the construction of $f$ from $\mathfrak{f}$ is independent of $t$, and $f$ is called the Shimura lift of $\mathfrak{f}$. Extending linearly from $\mathfrak{S}_{k+1/2}^{+,*}$ to $\mathfrak{S}_{k+1/2}^{+}$, we obtain the Shimura correspondence:
\begin{equation}\label{def:rho}
\rho : 
\begin{matrix}
\mathfrak{S}_{k+1/2}^+ \to \mathcal{H}_{2k}
\\\noalign{\vskip 1mm}
\hskip 5,6mm
\mathfrak{f} \mapsto f
\end{matrix}
\end{equation}
and $\rho$ gives an isomorphism between $\mathfrak{S}_{k+1/2}^+$ and $\mathcal{H}_{2k}$
(we shall use this fact many times). Note that $\rho$ depends on the choice of $\mathfrak{S}_{k+1/2}^{+,*}$, but it will not matter. Finally, Kohnen \cite{Kohnen1980} proved that $\rho$ is a finite linear combination of $\mathcal{S}_t$'s.

According to \cite[(1.18)]{Shimura1973}, for a Hecke eigenform $\mathfrak{f}$ of weight $k+\frac{1}{2}$ and any square-free positive integer $t$, the multiplicativity for its Fourier coefficients takes the form
\begin{equation}\label{multiplicativity}
{\mathfrak a}_{\mathfrak{f}}(tm^2) {\mathfrak a}_{\mathfrak{f}}(tn^2)
= {\mathfrak a}_{\mathfrak{f}}(t) {\mathfrak a}_{\mathfrak{f}}(tm^2n^2)
\quad{\rm if}\quad
(m, n)=1.
\end{equation}
If we write
\begin{equation}\label{def:a*f(n2)}
\mathfrak{a}_{\mathfrak{f}}^*(n^2)
:= \mathfrak{a}_{\mathfrak{f}}(t)^{-1} \mathfrak{a}_{\mathfrak{f}}(tn^2) n^{-(k-1/2)}
\end{equation}
and 
\begin{equation}\label{def:lambdaf(n)}
\lambda_f(n)
:= \mathfrak{a}_{\mathfrak{f}}(t)^{-1} a_{f_t}(n) n^{-(2k-1)/2},
\end{equation}
then the classical Hecke relation and \eqref{multiplicativity} imply that
the arithmetic functions $n\mapsto \lambda_f(n)$ and
$n\mapsto \mathfrak{a}_{\mathfrak{f}}^*(n^2)$ are multiplicative. 
With such notation, the formula \eqref{relation:aft-frakaft} can be written as 
\begin{equation}\label{relation:lambdaf(n)-fraka*ft(n)}
\lambda_f(n)
= \sum_{d\mid n} \frac{\chi_t(d)}{\sqrt{d}}
{\mathfrak a}_{\mathfrak{f}}^*\bigg(\frac{n^2}{d^2}\bigg).
\end{equation}
Since $f\in \mathcal{H}_{2k}^*$, $\lambda_f(n)$ is real and satisfies the Deligne inequality
\begin{equation}\label{Deligne:1}
|\lambda_f(n)|\le \tau(n)
\end{equation}
for all integers $n\ge 1$, where $\tau(n)$ is the classical divisor function (see \cite{Deligne1974}).
Let $\mu(n)$ be the M\"obius function.
Applying the M\"obius formula of inversion to \eqref{relation:lambdaf(n)-fraka*ft(n)}, we can derive that
\begin{equation}\label{relation:fraka*ft(n)-lambdaf(n)}
{\mathfrak a}_{\mathfrak{f}}^*(n^2)
= \sum_{d\mid n} \frac{\mu(d)\chi_t(d)}{\sqrt{d}} \lambda_f\bigg(\frac{n}{d}\bigg).
\end{equation}
Thus ${\mathfrak a}_{\mathfrak{f}}^*(n^2)$ is also real and \eqref{Deligne:1} implies that
\begin{equation}\label{Deligne:2}
|{\mathfrak a}_{\mathfrak{f}}^*(n^2)|\le \tau(n^2)
\end{equation}
for all integers $n\ge 1$.

Shimura's theory on modular forms of half-integral weight holds in general. To obtain Kohnen's isomorphism in general, one needs to develop a newform theory as Kohnen did in \cite{Kohnen1982} for the case of level $4N$ with $N$ square-free.

\section{Two large sieve inequalities}

This section is devoted to present two large sieve inequalities on eigenvalues of modular forms,
which will be one of the key tools in the proof of Theorem \ref{thm1}.
The first large inequality is related to modular forms of integral weights,
which is a particular case of \cite[Theorem 1]{LauWu2008} with $N=1$.

\begin{lemma}\label{lem2.1}
Let $\nu\ge 1$ be a fixed integer
and let $\{b_p\}_{p\in \P}$ be a sequence of real numbers
indexed by prime numbers such that $|b_p|\le B$ 
for some constant $B$ and for all prime numbers $p$.
Then we have
$$
\sum_{f\in \mathcal{H}_{2k}^*} 
\bigg|\sum_{P<p\le Q} 
b_p {\lambda_f(p^\nu)\over p}\bigg|^{2j}
\ll_\nu k 
\bigg({384B^2\nu^2j\over P\log P}\bigg)^j 
+ k^{10/11}\bigg({10BQ^{\nu/10}\over \log P}\bigg)^{2j}
\leqno(1.5)$$
uniformly for 
$$
B>0,
\qquad
j\ge 1,
\qquad
k\ge 3, 
\qquad
2\le P<Q\le 2P.
$$
The implied constant depends on $\nu$ only.
\end{lemma}

For modular forms of half-integral weights, we can prove the same large sieve inequality.

\begin{lemma}\label{lem2.2}
Let $\nu\geq 1$ be a fixed integer and let $ \{b_p\}_{p\in \P}$ be a sequence of real numbers indexed by prime numbers such that $|b_p|\leq B$ for some constant $B$ and for all prime numbers $p$. Then we have
$$
\sum_{\mathfrak{f}\in \mathfrak{S}_{k+1/2}^{+, *}}
\bigg|\sum_{P<p\leq Q} b_p \frac{\mathfrak{a}_{\mathfrak{f}}^*(p^{2\nu})}{p}\bigg|^{2j}
\ll_{\nu} k \bigg(\frac{1536B^{2}\nu^{2}j}{P \log P}\bigg)^{j} + k^{10/11}\left(\frac{20BQ^{\nu/10}}{\log P}\right)^{2j}
$$
uniformly for
$$
B>0, 
\qquad  
j\geq 1, 
\qquad  
k\ge 3,  
\qquad  
2\leq P < Q \leq 2P.
$$
The implied constant depends no $\nu$ only.
\end{lemma}

\begin{proof}
Taking $n=p^{\nu}$ in \eqref{relation:fraka*ft(n)-lambdaf(n)} gives us
\begin{equation}\label{relation:lambdaf(n)-fraka*ft(pnu)}
{\mathfrak a}_{\mathfrak{f}}^*(p^{2\nu})
= \lambda_f(p^{\nu}) - \frac{\chi_t(p)}{\sqrt{p}} \lambda_f(p^{\nu-1}).
\end{equation}
In view of the following facts that
$$
\bigg|\chi_t(p)\frac{\lambda_f(p^{\nu-1})}{\sqrt{p}}\bigg|\le \frac{\nu}{\sqrt{p}}
\qquad\text{and}\qquad
(|a|+|b|)^m\le (2|a|)^m + (2|b|)^m
$$
and of the Chebyshev estimate $\sum_{p\le x} 1\le 10x/\log x\;(x\ge 2)$, we can derive that
\begin{align*}
\bigg|\sum_{P<p\leq Q} b_p \frac{\mathfrak{a}_{\mathfrak{f}}^*(p^{2\nu})}{p}\bigg|^{2j}
& \le \bigg(
\bigg|\sum_{P<p\leq Q} b_p \frac{\lambda_f(p^{\nu})}{p}\bigg|
+ \bigg|\sum_{P<p\leq Q} b_p \chi_t(p) \frac{\lambda_f(p^{\nu-1})}{p^{3/2}}\bigg|
\bigg)^{2j}
\\
& \le 2^{2j}
\bigg|\sum_{P<p\leq Q} b_p \frac{\lambda_f(p^{\nu})}{p}\bigg|^{2j}
+ \bigg(\frac{20B\nu}{\sqrt{P}\log P}\bigg)^{2j}.
\end{align*}
Since the Shimura correspondence \eqref{def:rho} is a bijection between $\mathfrak{S}_{k+1/2}^{+, *}$ and $\mathcal{H}_{2k}^*$, 
we can write
$$
\sum_{\mathfrak{f}\in \mathfrak{S}_{k+1/2}^{+, *}}
\bigg|\sum_{P<p\leq Q} b_p \frac{\mathfrak{a}_{\mathfrak{f}}^*(p^{2\nu})}{p}\bigg|^{2j}
\le 4^j \sum_{f\in \mathcal{H}_{2k}^*}
\bigg|\sum_{P<p\leq Q} b_p \frac{\lambda_f(p^{\nu})}{p}\bigg|^{2j}
+ k \bigg(\frac{20B\nu}{\sqrt{P}\log P}\bigg)^{2j}.
$$
Now by applying Lemma \ref{lem2.1}, we have
\begin{align*}
\sum_{\mathfrak{f}\in \mathfrak{S}_{k+1/2}^{+, *}}
\! \bigg|\sum_{P<p\leq Q} b_p \frac{\mathfrak{a}_{\mathfrak{f}}^*(p^{2\nu})}{p}\bigg|^{2j} \!
& \ll_{\nu} k \bigg(\frac{1536B^{2}\nu^{2}j}{P \log P}\bigg)^{j} 
\! + k^{10/11}\left(\frac{20BQ^{\nu/10}}{\log P}\right)^{2j}
\! + k \bigg(\frac{20B\nu}{\sqrt{P}\log P}\bigg)^{2j}
\end{align*}
uniformly for
$B>0$, 
$j\geq 1$, 
$k\ge 3$
and 
$2\leq P < Q \leq 2P$.
This implies the required inequality since the third term on the right-hand side can be absorbed by the first one.
\end{proof}

\vskip 8mm

\section{Proof of Theorem \ref{thm1}}

Define
$$
\mathfrak{S}_{k+1/2}^{+, *}(P) 
:= \big\{\mathfrak{f}\in \mathfrak{S}_{k+1/2}^{+, *} : 
\varepsilon_p \mathfrak{a}_{\mathfrak{f}}^*(p^{2\nu})>0 \;\, \hbox{for} \;\,
p\in (P, 2P]\cap \mathcal{P}\big\}.
$$
It suffices to prove
that there are two positive constants 
$C=C(\nu, \mathcal{P})$ and $c=c(\nu, \mathcal{P})$ 
such that
\begin{equation}\label{(7.1)}
\big|\mathfrak{S}_{k+1/2}^{+, *}(P)\big|
\ll_{\nu} k \exp(-c(\log k)/\log_2k)
\end{equation}
uniformly for $k\ge k_0$ and $C\log k\le P\le (\log k)^{10}$
for some sufficiently large number $k_0=k_0(\nu, \mathcal{P})$.
 
For $1\le \mu\le \nu$, define 
$$
\mathfrak{S}_{k+1/2}^{+, *, \,\mu}(P)
:= \bigg\{\mathfrak{f}\in \mathfrak{S}_{k+1/2}^{+, *} : 
\bigg|\sum_{p\in (P, 2P]\cap \mathcal{P}} 
\frac{\lambda_f(p^{2\mu})}{p}\bigg|
\ge \frac{\delta}{4\nu\log P}\bigg\}.
$$
Take
$$\nu=2\mu,
\qquad 
Q=2P
\qquad{\rm and}\qquad
b_p = \begin{cases}
1 & \text{if $p\in \mathcal{P}$}
\\\noalign{\smallskip}
0 & \text{otherwise}
\end{cases}
$$ 
in Lemma \ref{lem2.2}. Then we get
\begin{align*}
\bigg(\frac{\delta}{4\nu\log P}\bigg)^{2j} 
\big|\mathfrak{S}_{k+1/2}^{+, *, \,\mu}(P)\big|
& \le \sum_{\mathfrak{f}\in \mathfrak{S}_{k+1/2}^{+, *}} 
\bigg|\sum_{P<p\le 2P} 
b_p \frac{\lambda_f(p^{2\mu})}{p}\bigg|^{2j}
\\
& \ll k \bigg(\frac{1536\mu^2 j}{P\log P}\bigg)^j 
+ k^{10/11}\bigg({10(2P)^{\mu/5}\over \log P}\bigg)^{2j}.
\end{align*}
Hence, 
\begin{equation}\label{(7.3)}
\big|\mathfrak{S}_{k+1/2}^{+, *, \,\mu}(P)\big|
\ll k \bigg({3456\nu^4j\log P\over \delta^2P}\bigg)^j 
+ k^{10/11} P^{\nu j},
\end{equation} 
provided $P\ge 200$. 

Let 
$$
b_p = \begin{cases}
\varepsilon_p  & \text{if $p\in \mathcal{P}$,}
\\\noalign{\smallskip}
0                     & \text{otherwise.}
\end{cases}
$$ 
From the definition of $\mathfrak{S}_{k+1/2}^{+, *}(P)$, \eqref{Deligne:2} and Lemma \ref{lem2.2}, 
we deduce that  
\begin{equation}\label{(7.4)}
\begin{aligned}
\sum_{\mathfrak{f}\in \mathfrak{S}_{k+1/2}^{+, *}(P)} 
\bigg|\sum_{p\in (P, 2P]\cap \mathcal{P}} 
\frac{\mathfrak{a}_{\mathfrak{f}}^*(p^{2\nu})^2}{p}\bigg|^{2j}
& \le (2\nu+1) \sum_{\mathfrak{f}\in \mathfrak{S}_{k+1/2}^{+, *}} 
\bigg|\sum_{P<p\le 2P} 
b_p \frac{\mathfrak{a}_{\mathfrak{f}}^*(p^{2\nu})}{p}\bigg|^{2j}
\\\noalign{\vskip -1mm}
& \ll_{\nu} k \bigg(\frac{1536\nu^{2}j}{P \log P}\bigg)^{j} + k^{10/11}\left(\frac{20Q^{\nu/10}}{\log P}\right)^{2j}
\\\noalign{\vskip 0,5mm}
& \ll_{\nu} k \bigg(\frac{1536\nu^{2}j}{P \log P}\bigg)^j 
+ k^{10/11} P^{\nu j/2}.
\end{aligned}
\end{equation}

In view of \eqref{relation:lambdaf(n)-fraka*ft(n)}, the Deligne inequality and the Hecke relation, it follows that
$$
{\mathfrak a}_{\mathfrak{f}}^*(tp^{2\nu})^2
\ge 1+\lambda_f(p^2)+\cdots+\lambda_f(p^{2\nu}) - 4\nu^2/\sqrt{p}.
$$
The left-hand side of \eqref{(7.4)} is 
\begin{align*}
& \ge \sum_{\mathfrak{f}\in \mathfrak{S}_{k+1/2}^{+, *}(P)\setminus 
\cup_{\mu=1}^{\nu} \mathfrak{S}_{k+1/2}^{+*, \, \mu}(P)} 
\bigg(
\sum_{\substack{P<p\le 2P\\ p\in \mathcal{P}}} \frac{1}{p}
- \sum_{1\le\mu\le\nu} \bigg|\sum_{P<p\le 2P} \frac{\lambda_f(p^{2\mu})}{p}\bigg|
- \frac{4\nu^2}{\sqrt{P}\log P}\bigg)^{2j}
\\
& \ge \sum_{\mathfrak{f}\in \mathfrak{S}_{k+1/2}^{+, *}(P)\setminus 
\cup_{\mu=1}^{\nu} \mathfrak{S}_{k+1/2}^{+*, \, \mu}(P)}  
\bigg(
\sum_{\substack{P<p\le 2P\\ p\in \mathcal{P}}} \frac{1}{p}
- \frac{\delta}{4\log P}
- \frac{4\nu^2}{\sqrt{P}\log P}
\bigg)^{2j}.
\end{align*}
Using the hypothesis \eqref{Hyp:mathcalP}, we infer that
\begin{align*}
\sum_{\substack{P<p\le 2P\\ p\in \mathcal{P}}} \frac{1}{p}
- \frac{\delta}{4\log P}
- \frac{4\nu^2}{\sqrt{P}\log P}
& \ge \frac{\delta}{\log P}
- \frac{\delta}{4\log P}
- \frac{\delta}{4\log P}
\\\noalign{\vskip -1,5mm}
& = {\delta\over 2\log P},
\end{align*}
provided $P\ge 256\nu^4\delta^{-2}$.

Combining these estimates with \eqref{(7.4)}, we conclude that
$$
\big|\mathfrak{S}_{k+1/2}^{+, *}(P)\setminus \cup_{\mu=1}^{\nu} \mathfrak{S}_{k+1/2}^{+*, \, \mu}(P)\big|
\ll_{\nu} k \bigg(\frac{1536\nu^{2}j\log P}{\delta^2P}\bigg)^j 
+ k^{10/11} P^{\nu j}.
$$
Together with \eqref{(7.3)}, it implies
\begin{equation}\label{(7.5)}
\big|\mathfrak{S}_{k+1/2}^{+, *}(P)\big|
\ll k \bigg({3456\nu^4j\log P\over \delta^2P}\bigg)^j 
+ k^{10/11} P^{\nu j}
\end{equation} 
uniformly for
$$
j\ge 1,
\qquad
2\mid k\ge 3,
\qquad
C\log k\le P\le (\log k)^{10}.
$$

Take
$$
j = \bigg[\delta^* \frac{\log k}{\log P}\bigg]
$$
where $\delta^*=\delta^2/(10(\nu+1))^4$. 
We can ensure $j>1$ once $k_0$ is chosen to be suitably large. 
A simple computation gives that  
$$
\bigg({3456\nu^4 j\log P\over \delta^2 P}\bigg)^j 
\ll \exp(-c(\log k)/\log_2k)
$$
for some positive constant 
$c=c(\nu, \mathcal{P})$ and
$P^{\nu j}
\ll k^{1/1000},
$
provided that $k_0$ is large enough.
Inserting them into \eqref{(7.5)}, we get \eqref{(7.1)} and complete the proof.
\hfill
$\square$

\vskip 8mm

\section{Proof of Theorem \ref{thm2}}

Since the proof of Theorem \ref{thm2} is rather similar to that of \cite[Theorem 4]{KLSW2010},
we shall only point out the differences.

It is well known that the \emph{Chebychev functions} $X_n$, $n\geq 0$, defined by
\begin{equation}\label{eq-cheby-pol}
X_n(\theta) := \frac{\sin((n+1)\theta)}{\sin\theta}
\qquad
(\theta\in [0, \pi])
\end{equation}
form an orthonormal basis of
$L^2([0,\pi],\mu_{\rm ST})$. Hence, for any integer $\omega\geq 1$, the functions
of the type
$$
(\theta_1,\ldots,\theta_{\omega})\mapsto \prod_{1\leq j\leq \omega}
{X_{n_j}(\theta_j)}
$$
for $n_j\geq 0$, form an orthonormal basis of
$L^2([0,\pi]^{\omega},\mu_{\rm ST}^{\otimes \omega})$.

For any $f\in \mathcal{H}_{2k}^*$ and prime $p$,
the Deligne inequality \eqref{Deligne:1} implies that there is a real number $\theta_f(p)\in [0,\pi]$ such that
\begin{equation}\label{def:thetafp}
\lambda_f(p) = 2\cos \theta_f(p).
\end{equation}

\begin{lemma}\label{lem4.1}
Let $k\in \N, s\in \N$ and $z\geq 2$ be a real number. 
For any prime $p\leq z$, let
$$
Y_p(\theta) := \sum_{0\le j\le s} {\hat{y}_{p}(j)X_j(\theta)}
$$
be a ``polynomial'' of degree $\leq s$ expressed in the basis of
Chebychev functions on $[0,\pi]$. Then we have
$$
\sum_{f\in \mathcal{H}_{2k}^*} \omega_f \prod_{p\leq z} Y_p(\theta_f(p))
= \prod_{p\leq z} {\hat{y}_p(0)} + O\big(C^{\pi(z)}D^{sz} k^{-5/6}\big),
$$
where $\|f\|^2$ is the Petersson norm of $f$,
$$
\omega_f := (4\pi)^{-(k-1)} \Gamma(k-1) \|f\|^{-2},
\qquad
C := \max_{p,j}{|\hat{y}_p(j)|}, 
$$
and $D\geq 1$ and the implied constant is absolute.
\end{lemma}

Let $z\ge 2$ be a parameter to be determined later and $L\equiv 3 \, (\text{mod}\,4)$ be a positive integer.
According to \cite[Theorem 7]{BMV00} with the choice of parameters $N=\pi(z)$ (the number of primes $p\le z$) and $u_n=0$, $v_n=\frac{1}{4}$ for
all $n\leq \pi(z)$, we can get two explicit trigonometric polynomials on $[0,1]^{\pi(z)}$, denoted
$A_L(\boldsymbol{x})$, $B_L(\boldsymbol{x})$, such that
\begin{equation}\label{BMV}
A_L(\boldsymbol{\theta}/\pi)-B_L(\boldsymbol{\theta}/\pi)
\leq \prod_{p\le z} \mathbb{1}_{[0, \frac{\pi}{4}]}(\theta_p) 
\end{equation}
for all $\boldsymbol{\theta} := (\theta_p)_{p\le z}\in [0,\pi]^{\pi(z)}$,
where $\mathbb{1}_{[0, \frac{\pi}{4}]}(t)$ is the characteristic function of 
$[0, \frac{\pi}{4}]\subset [0,\pi]$ 
(since $(v_n-u_n)(L+1)=\frac{1}{4}(L+1)$ is a positive
integer, we are in the situation $\Phi_{u,v}\in \mathcal{B}_N(L)$ of loc. cit.).
Moreover, $A_L(\boldsymbol{\theta}/\pi)$ is a product of
polynomials over each variable, and $B_L(\boldsymbol{\theta}/\pi)$ is a sum of $\pi(z)$ such products.

In view of \eqref{def:thetafp} and \eqref{relation:lambdaf(n)-fraka*ft(pnu)} with $\nu=1$, we have the following implicit relations
$$
\theta_f(p)\in [0, \tfrac{1}{4}\pi]
\;\Leftrightarrow\;
\lambda_f(p)\ge \sqrt{2}
\;\Rightarrow\;
\mathfrak{a}_{\mathfrak{f}}^*(p^2)\geq 0.
$$
Combining these with \eqref{BMV}, we can write, with the notation $\boldsymbol{\theta}_f=(\theta_f(p))_{p\le z}$,
\begin{equation}\label{proof-thm2:1}
\begin{aligned}
\sum_{\stacksum{\mathfrak{f}\in \mathfrak{S}_{k+1/2}^{+, *}} {\mathfrak{a}_{\mathfrak{f}}(p^2)\geq 0 \; \text{for} \; p\le z}} 
{\omega_f }
& \geq \sum_{\stacksum{f\in \mathcal{H}_{2k}^*}{\lambda_f(p)\geq \sqrt{2} \; \text{for} \; p\le z}} {\omega_f }
\\\noalign{\vskip -0,5mm}
& \geq \sum_{f\in \mathcal{H}_{2k}^*} 
\omega_f  \prod_{p\le z} \mathbb{1}_{[0, \frac{\pi}{4}]}(\theta_f(p))
\\\noalign{\vskip 0,5mm}
& \geq \sum_{f\in \mathcal{H}_{2k}^*} \omega_f 
\big(A_L(\boldsymbol{\theta}_f/\pi)-B_L(\boldsymbol{\theta}_f/\pi)\big).
\end{aligned}
\end{equation}

The next lemma is an analogue of \cite[Lemma 3.2]{KLSW2010}.

\begin{lemma}\label{lem4.2}
With notation as above, we have:
\par
\emph{(a)} 
For any $\varepsilon\in (0,\tfrac{1}{4})$, there exist constants $c>0$ and $L_0\geq 1$ 
such that the contribution $\Delta$ of the constant
terms of the Chebychev expansions of $A_L(\boldsymbol{\theta}/\pi)$ and
$B_L(\boldsymbol{\theta}/\pi)$ satisfies
$$
\Delta \geq (\tfrac{1}{4}-\varepsilon)^{\pi(z)},
$$
if $L\equiv 3\,({\rm mod}\,4)$ is the smallest integer $\geq \max\{c\pi(z), L_0\}$.
\par
\emph{(b)} 
All the coefficients in the expansion in terms of Chebychev
functions of the factors in $A_L(\boldsymbol{\theta}/\pi)$ or in the terms of
$B_L(\boldsymbol{\theta}/\pi)$ are bounded by $1$.
\par
\emph{(c)} 
The degrees, in terms of Chebychev functions, of the
factors of $A_L(\boldsymbol{\theta}/\pi)$ and of the terms of $B_L(\boldsymbol{\theta}/\pi)$,
are $\leq 2L$.
\end{lemma}

Take $L$ as in Lemma \ref{lem4.2}(a) (we can obviously assume $L\geq L_0$, since otherwise $z$ is bounded).
Since $A_L(\boldsymbol{\theta}/\pi)$ is a product of polynomials over each variable
and $B_L(\boldsymbol{\theta}/\pi)$ is a sum of $\omega$ such products, 
we can now apply Lemma \ref{lem4.1} to the terms on the right-hand side of \eqref{proof-thm2:1}.
Noticing that Lemma \ref{lem4.2}(b) implies $C\le 1$, we have
\begin{equation}\label{proof-thm2:2}
\begin{aligned}
\sum_{\substack{\mathfrak{f}\in \mathfrak{S}_{k+1/2}^{+, *}\\ \mathfrak{a}_{\mathfrak{f}}(p^2)\geq 0 \; \text{for} \; p\le z}} \omega_f 
& \geq \sum_{f\in \mathcal{H}_{2k}^*} \omega_f 
\big(A_L(\boldsymbol{\theta}_f/\pi)-B_L(\boldsymbol{\theta}_f/\pi)\big)
\\\noalign{\vskip -2,5mm}
& = \Delta + O(D^{z\pi(z)} k^{-5/6}).
\end{aligned}
\end{equation}

Fixing $\varepsilon\in (0, \tfrac{1}{8})$, taking $z = c_1\sqrt{(\log k)\log_2k}$ and using Lemma \ref{lem4.2}, we have
\begin{align*}
\Delta + O(D^{z\pi(z)} k^{-5/6})
& \ge (\tfrac{1}{4}-\varepsilon)^{\pi(z)} + O(D^{z\pi(z)} k^{-5/6})
\\
& \gg \exp\big(-(c_2/2)\sqrt{\log k)/\log_2k}\big).
\end{align*}
Combining it with \eqref{proof-thm2:2} and noticing that $\mathfrak{a}_{\mathfrak{f}}(p^2)\geq 0 \; \text{for} \; p\le z$
implies $n_{\mathfrak{f}}^*>z$, we find that
$$
\sum_{\substack{\mathfrak{f}\in \mathfrak{S}_{k+1/2}^{+, *}\\ n_{\mathfrak{f}}^*>c_1\sqrt{(\log k)\log_2k}}} \omega_f 
\gg \exp\big(-(c_2/2)\sqrt{\log k)/\log_2k}\big).
$$
Now the required result follows from this inequality thanks to the well-known bounds $\omega_f\ll (\log k)/k$.

\vskip 6mm

\noindent{\bf Acknowledgement}.
We began working on this paper during a visit of the first author at Universit\'e de Lorraine
during the academic year 2018-19.
He would like to thank the institute for the pleasant working conditions. 
This work is supported in part by the National Natural Science Foundation of China (Grant Nos. 11771121, 11971370 and 11871175).

\end{document}